\documentclass[a4paper, 12pt]{article}

\usepackage{amsmath, amsthm, amssymb}

\usepackage[utf8]{inputenc}
\usepackage[a4paper, margin=2.5cm]{geometry}
\usepackage{needspace}
\usepackage{accents}

\usepackage{libertine} 
\usepackage{inconsolata} 
\usepackage[T1]{fontenc}

\usepackage{graphicx}
\usepackage{enumerate}
\usepackage{authblk}
\usepackage[normalem]{ulem}
\usepackage{thmtools}
\usepackage{thm-restate}
\usepackage[colorlinks=true, citecolor=red]{hyperref}
\usepackage{float}
\usepackage{subcaption}

\DeclareMathOperator {\dist}{dist}

\usepackage[dvipsnames]{xcolor}
\usepackage{circuitikz}
\usepackage{cleveref}
\usepackage{comment}

\usepackage{todonotes}
\setuptodonotes{inline}

\renewenvironment{abstract}
{\small\vspace{-1em}
\begin{center}
\bfseries\abstractname\vspace{-.5em}\vspace{0pt}
\end{center}
\list{}{
\setlength{\leftmargin}{0.6in}%
\setlength{\rightmargin}{\leftmargin}}%
\item\relax}
{\endlist}

\declaretheorem[name=Theorem, numberwithin=section]{theorem}
\declaretheorem[name=Lemma, sibling=theorem]{lemma}

\declaretheorem[name=Definition, sibling=theorem]{definition}
\declaretheorem[name=Corollary, sibling=theorem]{corollary}

\declaretheorem[name=Claim, sibling=theorem]{claim}
\declaretheorem[name=Claim, numbered=no]{claim*}

\declaretheorem[name=Remark, style=remark, sibling=theorem]{remark}
\declaretheorem[name=Observation, style=remark, sibling=theorem]{observation}
\declaretheorem[name=Question, style=remark, sibling=theorem]{question}
\def\cqedsymbol{\ifmmode$\lrcorner$\else{\unskip\nobreak\hfil
\penalty50\hskip1em\null\nobreak\hfil$\lrcorner$
\parfillskip=0pt\finalhyphendemerits=0\endgraf}\fi}

\newcommand{\NN}{\mathbb{N}} 

\let\le\leqslant
\let\ge\geqslant
\let\leq\leqslant
\let\geq\geqslant

\title{On graph classes with constant domination-packing ratio}
\author[1]{Marthe Bonamy}
\author[2]{M\'onika Csik\'os}
\author[3,4]{Anna Gujgiczer}
\author[5]{Yelena Yuditsky}
\affil[1]{CNRS, LaBRI, Université de Bordeaux, Bordeaux, France.}
\affil[2]{Université Paris Cité, CNRS, IRIF}
\affil[3]{HUN-REN Alfréd Rényi Institute of Mathematics, Budapest, Hungary}
\affil[4]{MTA--HUN-REN RI Lendület ``Momentum'' Arithmetic Combinatorics Research Group, Budapest, Hungary}
\affil[5]{Université libre de Bruxelles (ULB), Brussels, Belgium}
\date{\today}

\begin{document}

\maketitle

\begin{abstract}
The dominating number $\gamma(G)$ of a graph $G$ is the minimum size of a vertex set whose closed neighborhood covers all the vertices of the graph. The packing number $\rho(G)$ of $G$ is the maximum size of a vertex set whose closed neighborhoods are pairwise disjoint. 
In this paper we study graph classes ${\cal G}$ such that $\gamma(G)/\rho(G)$ is bounded by a constant $c_{\cal G}$ for each $G\in {\cal G}$. We propose an inductive proof technique to prove that if $\cal G$ is the class of $2$-degenerate graphs, then 
there is such a constant bound $c_{\cal G}$.
We note that this is the first monotone, dense graph class that is shown to have constant ratio. We also show that the classes of AT-free and unit-disk graphs have bounded ratio. 
 In addition, our technique gives improved bounds on $c_{\cal G}$ for planar graphs, graphs of bounded treewidth or bounded twin-width. Finally, we provide some new examples of graph classes where  the ratio is unbounded.
\end{abstract}

\section{Introduction}\label{sec:intro}
The duality of packing and covering problems is well-known and their relation is widely studied in geometric setups. These two notions can naturally be translated to graphs using the notion of balls. 
\paragraph{Balls in graphs.} Let $G=(V,E)$ be a simple graph. The {\it distance} between two vertices $v,u\in V$, denoted $\dist_G(u,v)$, is the length of the shortest path between $u$ and $v$ in $G$.  Let $k\in \mathbb{N}$ and let a {\it $k$-ball} centered in $v\in V(G)$ be the set of all vertices at distance at most $k$ from $v$. 
For any $v\in V$ let $N_G[v]$ denote the $1$-ball centered in $v$, that is, the closed neighborhood of $v$ in $G$ and let $N_G(v)$ denote $N_G[v] \setminus \{v\}$. The {\it degree} of $v$ in $G$, denoted by $\deg_G(v)$ is defined as $|N_G(v)|$.
For any $X \subseteq V$ let $N_G[X]$ denote $ \bigcup_{v\in X} N_G[v]$ and let $N_G(X)$ denote $\bigcup_{v\in X} N_G(v)$. In the notations above and in subsequent ones, we drop the subscript when $G$ is clear from the context.

\paragraph{Covering and packing.} Covering a graph $G$ with as few balls of radius $1$ as possible is commonly known as finding a minimum dominating set of $G$. Formally, a {\it dominating set} in $G$ is a collection of vertices such that the union of their closed neighborhoods contains all of $V(G)$. We denote the size of minimum dominating set in $G$ by  $\gamma(G)$. 
The dual of this problem is to maximize the number of disjoint balls of radius $1$ that can be packed in a graph. 
Formally, a {\it packing} in $G$ is a collection of vertices whose closed neighborhoods are pairwise disjoint. In other words, a set of vertices form a packing iff their pairwise distances are at least $3$. We let $\rho(G)$ denote the size a maximum packing in $G$. \\

\noindent
A straightforward, general relation of $\gamma(G)$ and $\rho(G)$ is given by the following observation. For any graph $G$ and any packing, a dominating set must have a non-empty intersection with each of the closed neighborhoods of the vertices of the packing, thus $\gamma(G)\ge \rho(G)$. In general those two parameters are not equal. Simple examples where $\gamma(G) > \rho(G)$ include $G=C_4$ or $G=C_5$. 

\noindent
The problems of computing $\gamma(G)$ and $\rho(G)$ have natural formulations as integer programs. 
\begin{align}\label{eq:LP-relax}
&\text{minimize} \quad \sum_{v\in V(G)} x_v&
&\text{maximize} \quad \sum_{v\in V(G)} y_v \\
&\text{s.t.} \quad \sum_{u\in N[v]} x_u \ge 1, \mbox{     }\forall v\in V(G)&
&\text{s.t.} \quad  \sum_{u\in N[v]} y_u\le 1, \mbox{     }\forall v\in V(G), \nonumber\\
&\quad \quad x_v\in \{0,1\}, \quad \forall v\in V(G)&
&\quad \quad y_v\in \{0,1\}, \quad \forall v\in V(G) \nonumber
\end{align}
The linear program relaxations of the above integer programs are dual, and their optimums are called fractional dominating and fractional packing numbers, respectively. One of the early papers studying the integrality gap of these problems was due to Lov\'asz~\cite{Lovasz75}, who showed that the ratio of integral and fractional dominating numbers is always at most $\log (\Delta(G))$, where $\Delta(G)$ denotes the  maximum degree of $G$. 
In this paper, we study the following question.
\begin{restatable}{question}{Qmain}
    Given a graph class $\mathcal G$, is there a constant $c_{\mathcal G}$ such that \[\frac{\gamma(G)}{\rho(G)}\le c_{\mathcal{G}} \quad \text{ for each } \quad G\in \mathcal{G} ~~ ?\] 
\end{restatable}

\noindent
Importantly, a positive answer to this question immediately gives the bound $ c_{\mathcal{G}}$ on the integrality gap of both the dominating and the packing problem for any graph $G \in \mathcal G$.

\noindent

\subsection{Previous work}

There are graph classes for which the ratio $\gamma/\rho$ is known to be equal to $1$, namely, trees \cite[Theorem 7]{MM75}, strongly chordal graphs~\cite[Corollary 3.4]{Farber84}, and dually chordal graphs \cite[Theorem 3.2]{BCD98}.
The ratio is known to be at most $2$ for cactus graphs~\cite[Theorem 8]{LRR13} and connected biconvex graphs~\cite[Theorem 15]{GG24}.
Recently, it was also shown that for any bipartite cubic graph $G$, $\gamma(G)\le \frac{120}{49}\rho(G)$ holds \cite[Theorem 5]{GG24} and for any maximal outerplanar graph $H$, we have $\gamma(H)\le 3 \rho(H)$ \cite[Theorem 6]{GG24}.

\subsubsection{Low-degree graphs.} A general upper bound of $\gamma(G)\le \Delta(G)\rho(G)$  follows by observing that the union of all neighbors of the vertices in a maximum packing is a dominating set in the graph. This observation was first stated by Henning et al. \cite{HLR11}  who also showed that if $\Delta(G)\le 2$, then  $\gamma(G)\le \rho(G)+1$  with equality if and only if $G = C_n$ and $n \equiv 1, 2 \mod 3$.
 Furthermore, they proved that if $G$ is claw-free graph with $\Delta(G)\le 3$, then $\gamma(G)\le 2\rho(G)$ and stated the following conjecture on general subcubic graphs.

\begin{restatable}{conjecture}{maxdegconj}\label{conj:deg3}\cite{HLR11}
    Let $G$ be a connected graph with $\Delta(G)\le 3$, then $\gamma(G)\le 2\rho(G)+1$. The equality holds only for three  well-defined graphs, one of which being the Petersen graph.
\end{restatable}

\subsubsection{Graphs with no large complete bipartite minors.} A graph $H$ is a {\it minor} of a graph $G$ if it can be obtained from $G$ by deleting vertices and edges and by contracting edges. Böhme and Mohar~\cite{BM03} studied the more general problem of packing and covering with balls of radius $k$ in $K_{q,r}$-minor free graphs. Their results imply the following on the domination and packing numbers. 
\begin{theorem}\cite[Corollary 1.2]{BM03}\label{BMthm}
 If $G$ does not contain a $K_{q,r}$-minor, then $\gamma(G)< (4r+(q-1)(r+1))\rho(G)-3r+1$. 
\end{theorem}

\noindent
Since graphs embedded in a surface of genus $g$ cannot contain a $K_{k,3}$-minor for any $2g+3 \leq k$, the following bound holds.

\begin{corollary}\cite[Corollaries 1.3 and 1.4]{BM03}
   Let $G$ be a graph embedded in a surface of Euler genus $g$, then $\gamma(G)\le 4(2g+5)\rho(G)-9$. In particular, if $G$ is a planar graph, then $\gamma(G)\le 20\rho(G)-9$.    
\end{corollary}

\subsubsection{Graphs with bounded weak coloring numbers.}
\Cref{BMthm} was generalized in \cite{D13} using the {\it weak $1$-coloring} and {\it weak $2$-coloring numbers} ($\mbox{wcol}_1(G)$ and $\mbox{wcol}_2(G)$) of a graph $G$. 
Since in this paper we will not work directly with the notion of the weak coloring number, we spare its (lenghty) definition and only give the statement of the theorem.
\begin{theorem}\cite[Theorem 4]{D13}\label{Zdenekthm}
    For any $G$ such that $\mbox{wcol}_2(G)\le c$, we have $\gamma(G)\le c^2 \rho(G)$. 
\end{theorem}
\begin{remark}
Later, Dvo\v{r}\'ak proposed the alternative  bound of $\gamma(G)\le 4\mbox{wcol}_1^4(G) \mbox{wcol}_2(G) \rho(G)$, which is an improvement for the case where $\mbox{wcol}_2$ is large enough compared to $\mbox{wcol}_1$~\cite{D19}.  
\end{remark}
\Cref{Zdenekthm} implies that the ratio is bounded by some constant for many minor-closed and sparse graph classes, including planar graphs and bounded treewidth graphs. For more information on the weak coloring numbers of minor-closed graph classes, see~\cite{HLMR24}.

\subsubsection{Negative results.} There are much fewer graph classes that were proven to have unbounded ratio.
Burger et al.~\cite{BHV09} observed that if $G$ is the Cartesian product of two complete graphs on $n$ vertices, then $\gamma(G) = n$ and $\rho(G) = 1$. 
This  implies that there is no constant that bounds the $\gamma/\rho$ ratio for all bipartite graphs. Another class where the  ratio $\gamma/\rho$ can be arbitrarily large is the class of graphs with arboricity 3~\cite{D19}.

\subsection{Our results}

In this work we extend the lists of graph classes with bounded and unbounded $\gamma/\rho$ ratios. Our proofs showing bounded ratios rely on a unified inductive technique, outlined in Section~\ref{sec:stronger}. We summarize the reached state of the art in Figure~\ref{fig:classes}, where our results are highlighted with a grey background or grey border. We note that this figure was inspired by a similar one in~\cite{GG24} and that our results cover most of the classes listed as open problems in \cite{GG24}.

\begin{figure}[h!]
    \centering
    \includegraphics[width=\linewidth]{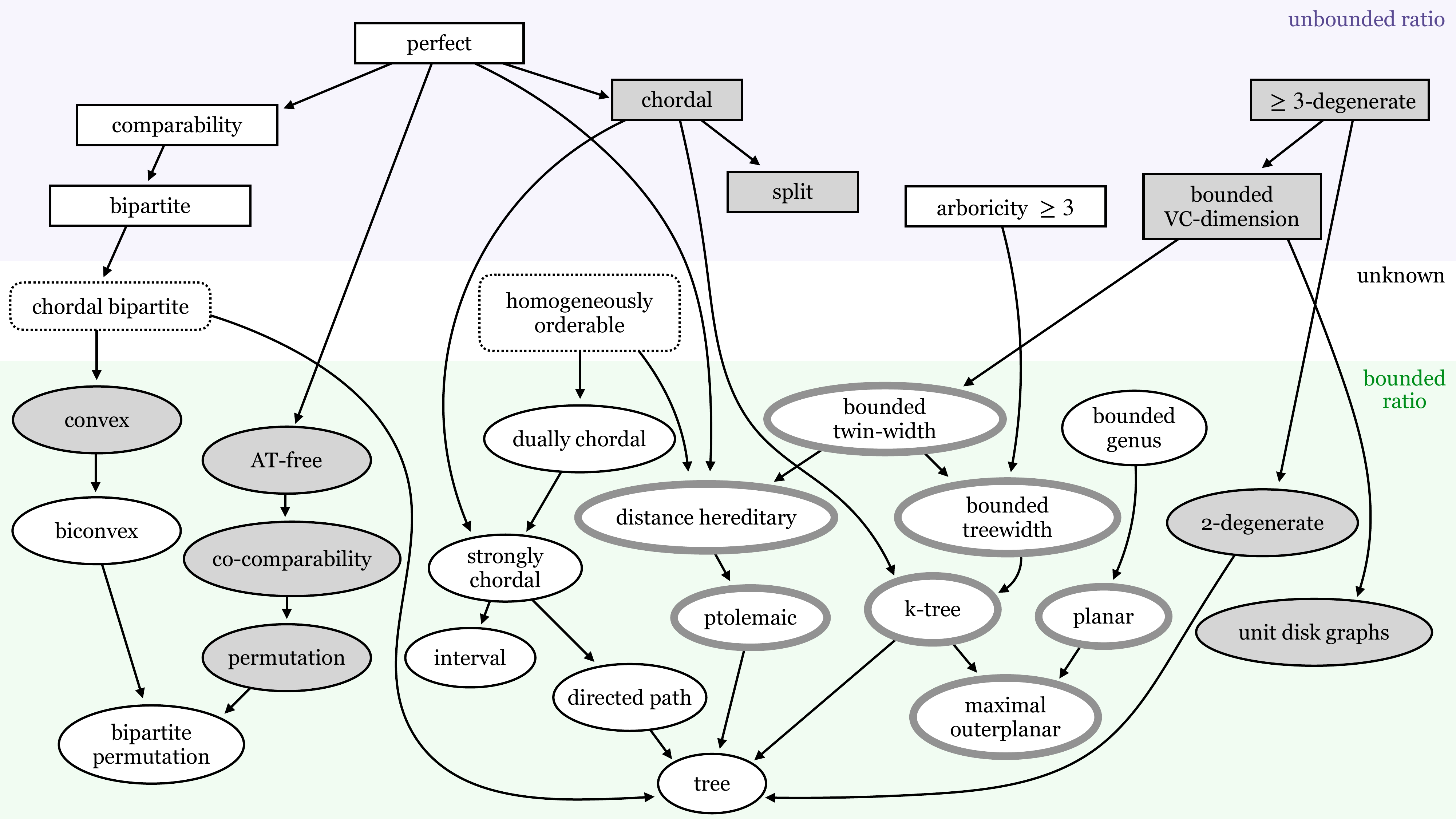}
    \caption{A visual map of graph classes for which the $\gamma/\rho$ ratio has been studied. The classes within rectangular boxes have unbounded ratio, the ones with ellipse boxes have bounded ratio and the dashed border indicates that the problem is not yet settled for the class. The boxes with the grey background highlight the classes for which we established bounds on the ratio and the thick grey borders indicate classes for which we improved the constant bound on the ratio.}
    \label{fig:classes}
\end{figure}
\subsubsection{New graph classes with bounded ratio}
Previous results have suggested that the key structural property for bounded ratio might be {\it sparsity}\footnote{The sparsity we refer to here is one which can, among other ways, be described using the weak coloring numbers.}.
We provide a positive results for the class of $2$-degenerate graphs which are known to be not sparse.

\begin{restatable}{theorem}{twodegenerate}
    For every $2$-degenerate graph $G$, we have $\gamma(G)\le 7 \cdot \rho(G)$. 
\end{restatable}

\noindent
We further extend the list of bounded ratio graph classes with asteroidal triple-free graphs, convex graphs, and intersection graphs of unit disks in the plane. 

\begin{restatable}{theorem}{ATfree}
    \label{th:ATfree}~
    \begin{enumerate}
        \item For every asteroidal triple-free graph $G$, we have $\gamma(G) \leq 3 \cdot \rho(G)$.
        \item For every convex graph $G$, we have $\gamma (G) \le 3 \cdot \rho (G)$.
        \item For every unit-disk graph $G$, we have $\gamma(G) \leq 32 \cdot \rho(G)$.        
    \end{enumerate}
    \end{restatable}
 
\noindent

\subsubsection{Improved constants}
The classes of graphs with bounded tree-width, planar graphs, and outerplanar graphs have constant weak coloring numbers, therefore by Theorem~\ref{Zdenekthm}, the $\rho/\gamma$ ratio is bounded by a constant for each of them.
    However, the precise constants obtained for these classes from the result in \cite{D13, D19} often not optimal for subclasses that have a lot of additional structure. We propose direct proofs for the above classes, resulting in very small constant bounds on the ratio.
For graphs with bounded tree-width (and therefore $k$-trees) we show the following bound. 

\begin{restatable}{theorem}{boundedTW}\label{thm:tw-dompack}
    Let $k\ge 1$ be an integer. For every graph $G$ with tree-width at most $k$, we have $\gamma(G)\le k\cdot \rho(G)$. 
\end{restatable}

Note that the above bound also improves the bound on $\gamma/\rho$ shown in \cite{GG24} from $3$ to $2$ where $G$ is an outerplanar graph. 
For planar graphs we show the following bound. 

\begin{restatable}{theorem}{planar}\label{thm:planar-dompack}
    For every planar graph $G$, we have $\gamma(G)\le 10 \cdot \rho(G)$. 
\end{restatable}

For the case of graphs with bounded twin-width, \cite[Section 7]{BGKTW24} showed that both LP-relaxations (given in Equation \eqref{eq:LP-relax}) have bounded integrality gap, which implies that the $\gamma/\rho$ ratio must also be bounded by a constant (but no explicit bound is given there). Here we give direct proof of the following statement.
\begin{restatable}{theorem}{twinwidth}\label{thm:tww-dompack}
    Let $k\ge 2$ be an integer. For every graph $G$ with twin-width at most $k$, we have $\gamma(G)\le 4k^2\cdot \rho(G)$. 
\end{restatable}

\subsubsection{Negative results}

We show that second part of \Cref{conj:deg3} (stated in \cite{HLR11} and studied in \cite{GG24}) does not hold. 

\begin{restatable}{theorem}{negativeConj}
    \label{thm:negativeConj}
    There is an infinite family of graphs ${\cal G}$ such that for each $G\in {\cal G}$, $\Delta(G)\le 3$ and $\gamma(G)=2\rho(G)+1$. 
\end{restatable}

We show that there is an infinite family of split graphs for which the ratio $\gamma/\rho$ is not bounded. This in particular implies that the above is not bounded for chordal graphs. 

\begin{restatable}{theorem}{negativeSplit}\label{th:negativesplit}
For every $k\ge 1$, there is a split graph $S_k$ such that $\rho(S_k)=1$ and $\gamma(S_k) = k$.
\end{restatable}

For $3$-degenerate graphs we present a construction that shows that the ratio is unbounded in general. This example also shows the existence of graphs with bounded VC-dimension and unbounded ratio. 

\begin{restatable}{theorem}{negativeThreedegenerate}\label{th:negative3deg}
For every $k\ge 1$, there is a $3$-degenerate graph $T_k$ such that $\rho(T_k) \leq 2$ and $\gamma(T_k) \geq k$.
\end{restatable}

\subsection*{Organization}
In \Cref{sec:stronger} we present the proof technique and the lemmas that will be used to prove most of our positive results. In Sections \ref{subsec:treewidth} and \ref{subsec:planar} we show improved bounds for graphs with bounded tree-width and planar graphs. In \Cref{sec:twinwidth} we establish that graphs with bounded twinwidth have bounded ratio and in \Cref{sec:kdeg} we discuss $k$-degenerate graphs. In \Cref{sec:additional} we list further graph classes for which we proved a positive result, and in \Cref{sec:negative} we prove all our negative results. Finally, we give our conclusion in \Cref{sec:ccl}. 

\section{A stronger induction hypothesis}\label{sec:stronger}

One of the difficulties with using induction to prove that ${\gamma}/{\rho}$ is bounded is that deleting vertices can affect the distance (in particular, deleting the common neighbor of two vertices and applying induction may result in them both being in the packing). We circumvent this by working with a stronger induction hypothesis, as follows.

\begin{definition}[$(X,Y)$-dominating set of $G$] 
    Let $G$ be a graph and let $X,Y\subseteq V(G)$. A set $D\subseteq V(G)$ is an \textit{$(X,Y)$-dominating set of $G$} if $N[D\cup X]\cup Y=V(G)$. We denote by $\gamma_{X,Y}(G)$ the size of a smallest $(X,Y)$-dominating set of $G$. 
\end{definition}

\begin{definition}[$(X,Y)$-packing of $G$]\label{packDef}
    Let $G$ be a graph and let $X,Y\subseteq V(G)$. A set $P\subseteq V(G)$ is an \textit{$(X,Y)$-packing of $G$} if any vertex $u \in P$ satisfies the following:
\begin{enumerate}[(a)]
    \item\label{packP} for any vertex $v \in P \setminus \{ u \}$, we have $\dist(u,v)\geq 3$.
    \item\label{packX} for any vertex $x \in X$, we have $\dist(u,x) \geq 2$, or equivalently $P\cap N[X]=\emptyset$.
    \item\label{packY} for any vertex $y \in Y$, we have $\dist(u,y) \geq 1$, or equivalently $P\cap Y=\emptyset$. 
\end{enumerate}
    We denote by $\rho_{X,Y}(G)$ the size of a largest $(X,Y)$-packing of $G$.  
\end{definition}

While this may seem opaque, the intuition is simple. Informally, $X$ is a set of vertices that will be ``for free'' in the dominating set and that may have a now-deleted neighbor in the packing, while $Y$ is the next layer: a set of vertices that are already dominated ``for free'' and that may be at distance two of a now-deleted vertex in the packing.

\paragraph{Proof technique.} Let $\mathcal{G}$ be a class of graphs and $C\in \mathbb{N}$. 
In order to prove that $\frac{\gamma(G)}{\rho(G)}\leq C$ for every graph $G \in \mathcal{G}$, we prove $\frac{\gamma_{X,Y}(G)}{\rho_{X,Y}(G)}\leq C$ for every graph $G \in \mathcal{G}$ and any choice of sets  $X,Y \subseteq V(G)$. Note that setting $X=Y=\emptyset$  yields the desired bound on the ratio $\gamma(G)/ \rho(G)$. 
Our inductive proofs have the following main steps.
\begin{itemize}
    \item Assume that there is a graph $G\in \mathcal{G}$ for which there are sets $X,Y\subseteq V(G)$ such that $\gamma_{X,Y}(G)/\rho_{X,Y}(G)>C$ and consider $G\in \mathcal{G}$ such that it minimizes $|V(G)|+|E(G)|$.
    \item Choose carefully a  $G'\in \mathcal G$ that is a subgraph of $G$ and sets $X', Y'\subseteq V(G')$. By the minimality of $G$, we know that $G'$ has an $(X',Y')$-dominating set  $D'$ and an $(X',Y')$-packing  $P'$---we will often compress this and say an  $(X',Y')$-dominating-packing pair $(D',P')$---that satisfy $|D'|/|P'|\le C$.
    \item Using $(D',P')$, define an $(X,Y)$-dominating-packing pair $(D,P)$ for $G$ that satisfies $|D|/|P|\le C$, reaching a contradiction. 
\end{itemize}

\subsection{Properties of $(X,Y)$-packings and $(X,Y)$-dominating sets}

The following lemmas state some properties of $(X,Y)$-packings and $(X,Y)$-dominating sets that will be particularly useful for induction purposes. Their proofs also follow the above outlined proof technique.

\begin{lemma}\label{lem:Xempty}
    Let $\mathcal{G}$ be a hereditary class of graphs and $C\in \mathbb{N}$. For any $G\in \mathcal{G}$ and $X,Y \subseteq V(G)$ such that $\frac{\gamma_{X,Y}(G)}{\rho_{X,Y}(G)}>C$ and $|V(G)|$ is minimal, we have $X=\emptyset$.
\end{lemma}

\begin{proof}
    Assume to the contrary that $X\neq \emptyset$ and let $a \in X$. 
    Define $X'=X\setminus \{a\}$, $Y'=Y \cup N(a)$, and $G'=G\setminus \{a\}$. 
    Furthermore, let $D'$ be the smallest $(X',Y')$-dominating set in $G'$ and let $P'$ be the largest $(X',Y')$-packing in $G'$. 
    By minimality of $G$, we have $|D'| \leq C\cdot |P'|$. 
    
    Firstly, note that $D=D'$ is an $(X,Y)$-dominating set in $G$. Indeed by the assumption on $D'$, $N_{G\setminus \{a\}}[D'\cup (X\setminus \{a\})]\cup (Y \cup N_G(a))=V(G)\setminus \{a\}$ and so $N_G[D\cup X]\cup Y=V(G)$ using the assumption that $a\in X$. 

    Secondly, we show that $P=P'$ is an $(X,Y)$-packing of $G$, by proving that it has all the properties listed in \Cref{packDef}. 
    Let $u,v\in P$, then by the properties of $P'$, $d_G(u,v)\ge 3$, unless $\{u,v\}\subseteq N(a)$, but this contradicts Property \eqref{packY} of $P'$, as $P'\cap Y'=P\cap (Y\cup N(a))=\emptyset$. Hence Property \eqref{packP} holds for $P$. 
    Let $u\in P$ and $x\in X$, then $d_G(u,x)\ge 2$ as $N(a)\cap P=\emptyset$ by Property \eqref{packY} of $P'$. Hence Property \eqref{packX} holds for $P$. 
    Let $u\in P$ and $y\in Y$, then as $P=P'$, Property \eqref{packY} holds for $P$. 

    These two observations imply that 
    $
    \frac{\gamma_{X,Y}(G)}{\rho_{X,Y}(G)} \leq \frac{|D|}{|P|} = \frac{|D'|}{|P'|}\leq C\, ,
    $
    reaching a contradiction to the assumption that $\frac{\gamma_{X,Y}(G)}{\rho_{X,Y}(G)}>C$ and thus we conclude $X=\emptyset$. 
\end{proof}

\begin{lemma}\label{lem:Ystable}
    Let $\mathcal{G}$ be a monotone class of graphs and $C\in \mathbb{N}$. For any $G\in \mathcal{G}$ and $X,Y \subseteq V(G)$ such that $\frac{\gamma_{X,Y}(G)}{\rho_{X,Y}(G)}>C$ and $|V(G)|+|E(G)|$ is minimal, we have that $G[Y]$ is an independent set.   
\end{lemma}

\begin{proof}
    Assume to the contrary that there is $y_1y_2 \in E(G[Y])$. Let $G'=G\setminus \{y_1y_2\}$, and let $D$ be a minimum $(X,Y)$-dominating set of $G'$ and $P$ be a maximum $(X,Y)$-packing of $G'$. By the minimality of $G'$, $|D| \leq C\cdot |P|$. 
        Observe that $D$ is an $(X,Y)$-dominating set in $G$ as $N_{G}[D\cup X]\cup Y = N_{G\setminus \{y_1y_2\}}[D\cup X]\cup Y=V(G)$.  
    Moreover, since $P\cap Y=\emptyset$, $P$ is an $(X,Y)$-packing of $G$ as well.
    Hence we get a contradiction to the assumption that $\frac{\gamma_{X,Y}(G)}{\rho_{X,Y}(G)}>C$ and conclude that $G[Y]$ is independent.
\end{proof}

\begin{lemma}\label{lem:smalldegY}
     Let $\mathcal{G}$ be a hereditary class of graphs and $C\in \mathbb{N}$. For any $G\in \mathcal{G}$ and $X,Y \subseteq V(G)$ such that $\frac{\gamma_{X,Y}(G)}{\rho_{X,Y}(G)}>C$ and $|V(G)|$ is minimal, we have that every vertex $u$ with $\deg(u) \leq C$ is contained in $Y$. 
\end{lemma}

\begin{proof}
    From the minimality of $G$ we know that $G$ is connected and by~\Cref{lem:Xempty}, $X=\emptyset$.
    Assume to the contrary that there is a vertex $a\in V(G)$ with $\deg(a) \leq C$ and $a \notin Y$.
    Let $G'=G\setminus\{a\}$, $X'=X\cup N(a)=N(a)$ and $Y' = Y$. We can assume that $N(a)\neq \emptyset$ as $G$ is connected. 
    By the minimality of $G$, there is an $(X',Y')$-dominating-packing pair $(D',P')$ in $G'$ such that $\frac{|D'|}{|P'|} \leq C$.
  
  Observe that $D=D'\cup N(a)$ is an $(X,Y)$-dominating set in $G$. By the assumption on $D'$, $N_{G\setminus \{a\}}[D'\cup (X \cup N(a))]\cup Y=V(G)\setminus \{a\}$. Therefore $N_G[(D'\cup N(a))\cup X]\cup Y= V(G)$, as required. 
Now we argue that  $P = P' \cup \{a\}$ is an $(X,Y)$-packing of $G$. Indeed, Property \eqref{packP} holds because of Property \eqref{packX} of $P'$. Property \eqref{packX} holds because $X=\emptyset$ and property \eqref{packY} holds because $a\notin Y$.

    We have $|D| \leq |D'| + C$ and $|P| = |P'| + 1$, hence $\frac{\gamma_{X,Y}(G)}{\rho_{X,Y}(G)}\le \frac{|D'|+C}{|P'|+1} \leq C$, a contradiction.
\end{proof}

\begin{lemma}\label{lem:y2deg}
     Let $\mathcal{G}$ be a hereditary class of graphs and $C\in \mathbb{N}$. Let $G\in \mathcal{G}$ for which there are $X,Y \subseteq V(G)$ be such that $\frac{\gamma_{X,Y}(G)}{\rho_{X,Y}(G)}>C$ and $|V(G)|$ is minimal, then for every vertex $y \in Y$ we have $\deg(y) \geq 2$.
\end{lemma}

\begin{proof}
    Assume to the contrary that there is $y\in Y$ such that $\deg(y)\le 1$. Let $G'=G\setminus \{y\}$, and $X'=X$, $Y'=Y\setminus\{y\}$. 
    
    Let $D'$ be the smallest $(X',Y')$-dominating set and $P$ the largest $(X',Y')$-packing of $G'$. By the minimality of $G$, $\frac{|D'|}{|P'|} \leq C$. Now notice that $D'$ is an $(X,Y)$-dominating set of $G$, since $y\in Y$, and $P'$ is an $(X,Y)$-packing of $G$, as adding $y$ cannot change the distances between any two vertices, because $\deg(y)\leq 1$. Again, we came to a contradiction with $\frac{\gamma_{X,Y}(G)}{\rho_{X,Y}(G)}>C$.
\end{proof}

\section{Graphs with bounded tree-width}\label{subsec:treewidth}

There are many equivalent ways to define the tree-width of a graph, we use the following one. The {\it tree-width} of a graph is $\min \{\omega(H)-1\mid G \mbox{ is a subgraph of } H \mbox{ and } H \mbox{ is chordal}\}$, where $\omega(H)$ denotes the order of the largest clique in $H$ \cite[Proposition 12.4.4]{Dbook}.
We prove the following stronger statement, which then implies \Cref{thm:tw-dompack} by substituting $X = Y = \emptyset$.

\begin{theorem}\label{th:treewidth}
    For every $k$, for every graph $G$ with tree-width at most $k$ and any sets $X,Y \subseteq V(G)$, we have $\gamma_{X,Y}(G) \leq k \cdot \rho_{X,Y}(G)$.
\end{theorem}

\begin{proof}
We proceed by contradiction. Consider a counterexample $(G,X,Y)$ which minimizes $|V(G)|+|E(G)|$. By Lemmas~\ref{lem:Xempty} and~\ref{lem:Ystable}, we may assume that $X=\emptyset$ and that $G[Y]$ induces a stable set.

By our assumptions, $G$ has tree-width at most $k$, therefore $G$ is a subgraph of some chordal graph $G^+$ on the same vertex set with $\omega(G^+) \leq k+1$. 
Since $G^+$ is chordal, it contains at least one simplicial vertex, that is, a vertex whose neighborhood in $G^+$ induces a clique. Let 
\[A_1 = \{v\in V : v \text{ is simplicial in } G^+ \}
\quad \text { and } \quad
A_2 = \{v\in V : v \text{ is simplicial in } G^+\setminus A_1 \}.\]
We have that $A_1\neq \emptyset$ and that each $v\in A_2$ has at least one neighbor in $A_1$, as otherwise it would be contained in $A_1$. 
Moreover, $\omega(G^+) \leq k+1$ implies that for any $v \in A_1$, $\deg_{G^+}(v)\leq k$. As $G^+$ is obtained from $G$ by adding edges, we also have that $\deg_{G}(v)\leq k$ for each $v \in A_1$. Hence by Lemma~\ref{lem:smalldegY}, we get that $A_1\subseteq Y$.

Since $G^+ \setminus A_1$ is also chordal, $A_2 = \emptyset$ if and only if $G^+ \setminus A_1$ is the empty graph, in particular $V(G^+) = V(G) = A_1 \subseteq Y$
However, in this case, $\emptyset$ is an $(X,Y)$-packing and $(X,Y)$-dominating set of size $0$, and so $\gamma_{X,Y}(G) \leq k \cdot \rho_{X,Y}(G)$ is satisfied, contradicting that $(G,X,Y)$ is a counterexample.
Therefore we can assume that $A_2\neq \emptyset$ and consider 
\[
\text{a vertex $v\in A_2$, \quad the sets $B = N_{G^+}(v) \cap A_1 \neq \emptyset$, \quad  $C = N_{G^+}(v) \setminus A_1$, \quad and $u \in B$. }
\]
\noindent
Observe that as $u$ is simplicial in $G^+$, we have $N_{G^+}(u) \subseteq N_{G^+}[v] = B \cup C \cup \{v\}$ and so $N_{G}(u) \subseteq N_{G^+}(u) \subseteq B \cup C \cup \{v\}$. On the other hand, $B\subseteq A_1\subseteq Y$ which is stable in $G$ and so $N_{G}(u)  \subseteq C \cup \{v\}$.
Now we consider the sets
\[
C_1 =  C \cap N_G(v) \quad \text{ and } \quad C_2 = \{ c \in C ~:~ \dist_G(v,c)=2\} \subseteq C\setminus C_1,
\]
and define $G' = G \setminus \{v\}$, $X' =C_1$, and $Y'= Y\cup C_2$. By the minimality of $G$, $G'$ has an $(X', Y')$-dominating set $D'$  and an $(X', Y')$-packing $P'$ such that $|D'| \leq k\cdot |P'|$. 

First we argue that $P = P' \cup \{v\}$ is an $(X,Y)$-packing in $G$ by verifying Properties \eqref{packP}-\eqref{packY} of Definition~\ref{packDef}. Property \eqref{packX} trivially holds as $X=\emptyset$ and Property \eqref{packY} holds because $P' \cap Y' = \emptyset$ and $v\notin Y$. 
To establish Property \eqref{packP}, since $P'$ is a packing in $G \setminus \{v\}$, it is sufficient to show that $\dist_G(p,v)\ge 3$ for all $p \in P'$.
Assume to the contrary that there is $p \in P'$ such that $\dist_G(p,v)\le 2$, or equivalently, $N_G(v)\cap N_G[p] \neq \emptyset$ and let $u\in N_G(v)\cap N_G[p]$. Recall that $N_G(v) \subseteq B \cup C_1$ and that $P'$ is a $(C_1, Y\cup C_2)$-packing, which together imply  $N_G[p] \cap C_1 = \emptyset$. Therefore it must be the case that $u\in B$ and so $p \in N_G[u] \subseteq C \cup \{v\} \cup \{u\}$. 
As $p$ is a vertex of  a $(C_1, Y\cup C_2)$-packing in $G\setminus \{v\}$, we have $p \neq v$, $p \notin N_{G'}[C_1]$, and $p \notin Y \cup C_2 \supseteq B$. However as $\dist_G(p,v)\le 2$, it must be that $p \in B \cup C_1 \cup C_2 \cup \{v \} $, a contradiction.

To define an $(X,Y)$-dominating set $D$ of $G$, we distinguish two cases.
If $C_1\cup C_2\neq \emptyset$, we set $D=D'\cup C_1\cup C'_2$ where $C'_2$ is the set of common neighbors of $v$ and some vertex in $C_2$. 
Note that since $\deg_G(v) \leq k$ and $C_1 \cup C'_2 \subseteq N_G(v)$, we have $|C_1\cup C'_2|\le k$.
Since $D'$ is an $(X',Y')$-dominating set in $G'$, we have that 
$N_{G'}[D'\cup C_1]\cup (Y\cup C_2)=V(G')=V(G)\setminus \{v\}.$
Therefore 
$N_G[D'\cup C_1\cup C'_2] \cup Y=V(G)$, that is,
 $D$ is an $(X,Y)$-dominating set of $G$ of size $|D| \leq |D'| + k \leq k|P'| + k = k|P|$, a contradiction.
 
If $C_1\cup C_2=\emptyset$, then we set $D=D'\cup \{v\}$ and again get an $(X,Y)$-dominating set of $G$ (recall that $X = \emptyset$) of size $|D| \leq |D'| + 1 \leq k|P'| + 1 < k|P|$, a contradiction.
\end{proof}

\section{Planar graphs}\label{subsec:planar}

As before, we shall prove the following statement on $(X,Y)$-dominating sets and $(X,Y)$-packings, which implies \Cref{thm:planar-dompack} by taking $X=Y=\emptyset$.

\begin{theorem}\label{th:planar}
    For every planar graph $G$ and any sets $X,Y \subseteq V(G)$, we have $\gamma_{X,Y}(G) \leq 10 \rho_{X,Y}(G)$.
\end{theorem}

\begin{proof}
We proceed by contradiction. Consider a counterexample $(G,X,Y)$ which minimizes $|V(G)|+|E(G)|$. By Lemmas~\ref{lem:Xempty},~\ref{lem:Ystable}, and~\ref{lem:y2deg}, we have that $X=\emptyset$, $G[Y]$ induces a stable set, and for every vertex $y\in Y$, $\deg(y) \geq 2$.

Let $\pi$ be a planar embedding of $G$. For each vertex in $y\in Y$, consider the clockwise ordering of the neighbors of $y$ in the embedding $\pi$. For every two consecutive neighbors $v_i, v_{i+1}$ of $y$, if $v_iv_{i+1} \notin E(G)$, then we add the edge $v_iv_{i+1}$ inside the face which contains $y, v_i$, and $v_{i+1}$ on its boundary. In the special case where $y$ only has only two neighbors $v_1,v_2$, we add a parallel edge $v_1v_2$ and count the neighbors of these vertices with multiplicity.  See Figure \ref{fig:planar+} for an illustration of the edge additions around some vertex $y\in Y$. 

\begin{figure}[h]\centering\label{fig:planar_supergraph}
    \begin{subfigure}{.5\textwidth}
      \centering
      \includegraphics[width=.7\linewidth]{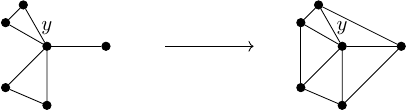}
    \end{subfigure}%
    \begin{subfigure}{.5\textwidth}
      \centering
      \includegraphics[width=.6\linewidth]{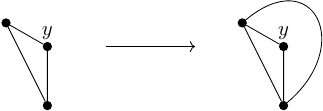}
    \end{subfigure}
    \caption{Edge additions in $G^+$.\label{fig:planar+}}
\end{figure}

\noindent 
Let $G^+$ be the resulting graph. Note that $G^+$ is planar, $G^+[Y]$ is an independent set (as $G[Y]$ was independent and we only added edges between neighbors of vertices in $Y$), and $G^+$ has the following property
\begin{enumerate}
    \item[($\star$)] every vertex $v \in V(G)$ has at most half of its neighbors in $Y$.
\end{enumerate}

\noindent
Property ($\star$) can be deduced as follows. If $v \in Y$, the statement is trivial by (ii).
Let $v\notin Y$, and assume that $v$ has a neighbor $y\in Y$. 

Let $f_1, f_2$ be the two faces of $G^+$ in its planar embedding which contain the edge $vy$ on its boundary. By the construction of $G^+$, each $f_i$ ($i = 1,2$) contains an edge connecting $v$ to some $u_i \in N_G(y)$ (we might have $u_1 = u_2 = u$, but then $uv$ must be a multi-edge
and we count $u$ as a neighbor twice). Note that since $G^+[Y]$ is independent, we have $u_i \notin Y$. Thus, for each $y \in Y \cap N_G(v)$, we can find at least two neighbors of $v$ that are not in $Y$. On the other hand, each edge $uv$ is contained in at most two faces of the planar embedding of $G^+$, therefore there are at most two $y_1, y_2 \in Y \cap N_G(v)$ that forms a face together with $uv$.\\

\noindent
We consider the sets  
$A_1=\{v\in V(G): d_{G^+}(v)\le 5\}$ and $A_2=\{v\in V(G): d_{G^+\setminus A_1}(v)\le 5\}$.
Note that since $G^+$ is planar, it must contain a vertex of degree at most $5$ and so $A_1 \neq \emptyset$. 
 As  we obtained $G^+$ from $G$ by only adding edges, we have $A_1\subseteq \{v\in V(G):d_G(v)\le 5\}$.  On the other hand, by \Cref{lem:smalldegY}, we have $\{v\in V(G):d_G(v)\le 10\}\subseteq Y$, and so $A_1 \subseteq Y$.
If $A_2=\emptyset$, then $G^+ \setminus A_1$ is the empty graph and so $V(G) = A_1 \subseteq Y$. However, this would imply that  $\gamma_{X,Y}(G)=0$ and $\rho_{X,Y}(G)=0$, a contradiction.

Therefore we can assume that $A_2 \neq \emptyset$. First observe that for any vertex $v \in A_2$, we have $N_{G^+}(v)\cap A_1\neq \emptyset$ as otherwise $v\in A_1$. Since $A_1 \subseteq Y$ and $G^+[Y]$ is independent, $A_2\cap Y=\emptyset$ follows. 
Next we argue that for every $v\in A_2$, $d_{G^+}(v)\le 10$. Indeed, as $v \in A_2$, $d_{G^+\setminus A_1}(v)\le 5$ and by Property ($\star$), at most half of its neighbors are from $A_1\subseteq Y$. 
Therefore, we conclude  $A_2 \subseteq \{v\in V(G): d_{G^+}(v)\le 10\} \subseteq \{v\in V(G): d_{G}(v)\le 10\} \subseteq Y$, a contradiction to $A_2 \cap Y = \emptyset$. 
\end{proof}

\section{Bounded twin-width graphs}\label{sec:twinwidth}

Before we discuss how to prove that graphs with bounded twin-width admit a bounded ratio $\gamma/\rho$, let us consider the case of \textit{distance-hereditary graphs} as a gentle warm-up. 
We say that a graph $G$ is distance-hereditary if and only if for every connected induced subgraph
$H$ of $G$, the distance between every pair of vertices in $H$ is the same as in $G$. 
We will prove the following theorem.

\begin{restatable}{theorem}{distance-h}
\label{th:distance-h-general}
    For every distance-hereditary graph $G$, we have $\gamma(G) \leq 2 \cdot \rho(G)$.
\end{restatable}

The proof will contain similar but simpler ingredients, as distance-hereditary graphs have small twin-width \cite{o05,BKTW21}. As before, we will work with $(X,Y)$ dominating sets and packings, but now we will use an even stronger notion $\gamma^t_{X,Y}(G)$, which requires total domination. Formally,  
$D$ is a total $(X,Y)$-dominating set if $D$ is an $(X,Y)$-dominating set and every vertex not in $X\cup Y$ has a neighbor in $D$ (even if it is itself in $D$) unless it has no neighbor; and $\gamma^t_{X,Y}(G)$ denotes the size of the smallest such set $D$.

    \begin{theorem}\label{lem:total}
        For every distance-hereditary graph $G$ and any sets $X,Y \subseteq V(G)$, we have $\gamma^t_{X,Y}(G) \leq 2 \rho_{X,Y}(G)$.
    \end{theorem}

\noindent
Note that $\gamma_{X,Y}(G)  \leq \gamma^t_{X,Y}(G)$ by definition, therefore by applying \Cref{lem:total} with $X = Y = \emptyset$, we recover \Cref{th:distance-h-general}.

    \begin{proof}[\textbf{Proof of \Cref{lem:total}}]
        We proceed by contradiction. Consider a counterexample $(G,X,Y)$ which minimizes $|V(G)|$. Note that similarly to Lemma~\ref{lem:Xempty} we may assume $X=\emptyset$. By \cite{BM86}, any distance-hereditary graph $G$ contains a vertex $u$ of degree $1$ or admits two vertices $u,v$ such that either $N(u)=N(v)$ or $N[u]=N[v]$. 
        
        If $G$ contains a vertex $u$ of degree $1$, we note that $u \not\in X \cup Y$ as $X = \emptyset$ and similarly to Lemma~\ref{lem:y2deg} any vertex in $Y$ that does not have at least two neighbors not in $X \cup Y$ may be deleted without impacting the solution. Let $v$ be the unique neighbor of $u$. We let $D$ be a smallest total $(\{v\},Y)$-dominating set of $G \setminus \{u\}$ and $P$ be a largest $(\{v\},Y)$-packing of $G \setminus \{u\}$. By the minimality of $G$, we have $|D|\leq 2|P|$. However, the set $D\cup\{u,v\}$ is a total $(\emptyset,Y)$-dominating set of $G$ and the set $P \cup \{u\}$ is an $(\emptyset,Y)$-packing in $G$, a contradiction. 
        
        Therefore, $G$ contains two vertices $u,v$ with either $N(u)=N(v)$ or $N[u]=N[v]$. Note that neither $u$ nor $v$ belongs to $X \cup Y$, as $X=\emptyset$ and if $u \in Y$ (resp. $v \in Y$) then $G \setminus \{u\}$ (resp. $G \setminus \{v\}$) would also be a counterexample, contradicting the minimality of $G$. Let $A=N(u)\cap N(v)$. Since we already proved that $G$ contains no vertex of degree $1$, we know that $A$ is not empty. We consider $G\setminus \{u\}$ and let $D$ be a smallest total $(\emptyset,Y)$-dominating set and $P$ be a largest $(\emptyset,Y)$-packing of $G \setminus \{u\}$. By the minimality of $G$, we have $|D|\leq 2|P|$. Note that $N(u) = N(v)$ implies that $P$ is an $(\emptyset,Y)$-packing of $G$. Moreover, $D$ is a total $(\emptyset,Y)$-dominating set of $G$, indeed $D$ contains at least one element of $N_G(u)$ since it contained at least one vertex of $N_{G \setminus \{u\}}(v) \subseteq N_G(u) $, where $N_{G \setminus \{u\}}(v) \neq \emptyset$ as $A \neq \emptyset$. In conclusion, $P$ is an $(\emptyset,Y)$-packing of $G$ and $D$ is a total $(\emptyset,Y)$-dominating set of $G$ such that $|D|\leq 2|P|$, a contradiction.           
    \end{proof}

\noindent
Now we turn to the problem of domination and packing in graphs of bounded twin-width.

\paragraph{Definition of twin-width.}Let $G$ be a graph and $R$ be a subset of its edges, which we will refer to as its \textit{red} edges (by contrast, edges in $E(G)\setminus R$ will be referred to as \textit{black} edges). \textit{Contracting} a pair $u,v$ of vertices in $G$ results in a graph $G'=G|_{u,v \rightarrow w}$ obtained from $G$ by deleting $u$ and $v$ and adding a vertex $w$ that has a black edge to any vertex that had a black edge to both $u$ and $v$, and a red edge to any vertex in $N_G(u) \cup N_G(v)$ that does not fall in the previous category.
    A graph $G$ with only black edges has \textit{twin-width} at most $k$ if and only if there is a series of contractions from $G$ to $K_1$ such that, at any step, the red edges form a graph with maximum degree at most $k$. By extension, we may say that a graph with some red edges has twin-width at most $k$ if there is a series of contractions from this graph to $K_1$ such that, at any step, the red edges form a graph with maximum degree at most $k$. \\
    
\noindent
Once again, we do not prove \Cref{thm:tww-dompack} directly, but rather work with a different induction hypothesis. This time, our hypothesis will use a stronger notion of domination number, but only allow $X = \emptyset$. Let $G$ be a graph with edges colored red and black. Let $e=uv\in E(G)$, we say that $u$ (resp. $v$) is a black neighbor of $v$ (resp. $u$) if $uv$ is black and it is a red neighbor of $uv$ is red. 
    We will use a stronger notion of domination number, denoted $\gamma^B_{\emptyset,Y}(G)$, which requires domination via black edges whenever possible. Formally, $D$ is a B-$(\emptyset,Y)$-dominating set if it is a $(\emptyset,Y)$-dominating set and every vertex not in $Y$ has a black neighbor in $D$ (even if it is itself in $D$) unless it has no black neighbor (in which case we just require it to be contained in $N[D]$). The notion of $(X,Y)$-packings remains unchanged in edge-colored graphs, with each color having an equal role.

        \begin{lemma}\label{lem:twinw}
        For every graph $G$ with red edges $R$ such that $(G,R)$ has twin-width at most $k$, for any set $Y \subseteq V(G)$, we have $\gamma^B_{\emptyset,Y}(G) \leq 4 k^2 \rho_{\emptyset,Y}(G)$.
    \end{lemma}

\noindent
  Note that $\gamma^B_{\emptyset,Y}(G) \geq \gamma_{\emptyset,Y}(G) $, therefore we can deduce \Cref{thm:tww-dompack} by applying Lemma~\ref{lem:twinw} to $G$ with $R=X=Y=\emptyset$.
    
    \begin{proof}[\textbf{Proof of \Cref{lem:twinw}}]
        We proceed by contradiction. Consider a counterexample $(G,R,Y)$ which minimizes $|V(G)|$.

        \begin{claim*}
             Any vertex $u \in V(G)$ with at most $k$ black neighbors belongs to $Y$.
        \end{claim*}
        \begin{proof}
        Consider such a vertex $u$ and let $B_u$ denote the set of its black neighbors and $R_u$ denote the set of its red neighbors. Note that $|N_G(u)| = |B_u \cup R_u| \leq 2k$.
        We consider the second-neighborhood $S$ of $u$, and partition it into two sets:
        $S_B$ containing the vertices of $S$ with some black edge to $B_u \cup R_u$, and $S_R$ containing those vertices of $S$ that have only red edges to $B_u \cup R_u$. Note that since every vertex is incident to at most $k$ red edges, we have $|S_R|\leq k  |B_u \cup R_u| \leq 2k^2$.

        We apply induction on $(G \setminus (B_u \cup R_u \cup \{u\}), Y \cup S_R \cup S_B)$ and obtain a B-$(\emptyset,Y \cup S_R \cup S_B)$-dominating set $D'$ and an $(\emptyset, Y \cup S_R \cup S_B)$-packing $P'$, with $|D'| \leq 4 k^2  |P'|$. If $u \not\in Y$, we consider $P= P' \cup \{u\}$ and note that $P$ is a $(\emptyset,Y)$-packing of $G$. Then, we consider $D$ obtained from $D'$ as follows. We first add all of $B_u \cup R_u$ to $D$, then make sure to B-dominate vertices in $S_R$ (note that vertices in $S_B$ are B-dominated by definition) by adding for each element in $S_R$ either a black neighbor or itself if it has none. Finally, we make sure to B-dominate vertices in $B_u \cup R_u$ by adding $u$ and for each element in $R_u$ a black neighbor if it has one. This way, we obtained a black $(\emptyset,Y)$-dominating set $D$ of $G$ of size 
        \begin{align*}
            |D| 
            &\leq 
            |D'|+|B_u|+|R_u|+|S_R|+|\{u\}|+|R_u|
            \leq 
            |D'|+k+k + 2k^2+1 +k \\
            &\leq 
            4k^2|P'|+2k^2+3k+1
            \leq 
            4k^2(|P'| +1)
            =
            4k^2|P|            
        \end{align*}
        Therefore, we obtain a contradiction unless $u \in Y$.
        \end{proof}

        \noindent
        Since $(G,R)$ has twin-width at most $k$, there are two vertices $u,v \in V(G)$ such that contracting $u$ and $v$ in $G$ results in a graph with twin-width at most $k$. We consider $(G',R')$ to be the graph obtained by contracting $u$, $v$ to a vertex $w$. If $u$ and $v$ both belong to $Y$, we set $Y'=Y \setminus \{u,v\} \cup \{w\}$. Otherwise, we consider without loss of generality $u \not\in Y$ and set $Y'=Y \setminus \{v\}$. By minimality, $(G',R',Y')$ admits a B-dominating set $D$ and a packing $P$ such that $|D|\leq 4 k^2 |P|$. The packing $P$ is also a packing in $(G,R,Y)$, free to replace $w$ with $u$. We now argue that $D$ is a B-$(\emptyset,Y)$-dominating set of $G$, again up to replacing $w$ with $u$. Indeed, for every vertex $x$ with a black neighbor in $G'$, there is a vertex $y \in D$ such that the edge $xy$ is black in $G'$ and hence black in $G$. We are concerned with the case of a vertex $x$ with no black neighbor in $G'$. However, since $x$ has no black neighbor in $G'$, it has at most $k$ black neighbors in $G$ hence belongs to $Y$. Therefore, $D$ is a B-$(\emptyset,Y)$-dominating set of $(G,R)$. Since $|D|\leq 4 k^2 |P|$, we reach a contradiction and conclude the proof of \Cref{lem:twinw}.
    \end{proof}

\section{$k$-degenerate graphs}\label{sec:kdeg}

A graph $G$ is $k$-degenerate if every subgraph of $G$ has a vertex of degree at most $k$.
In this section we show that the $\gamma/\rho$ ratio is  bounded by a constant for $2$-degenerate graphs, however it is not the case for higher degeneracy. We start by giving the proof of the second result.

\negativeThreedegenerate* 

\begin{proof}
    For $k\in \mathbb{N}$, we consider the graph $T_k$ defined as follows. The vertex set of $T_k$ is the disjoint union of three sets $A,B,\{v\}$ such that $|A|=2k$ and $|B|=\binom{2k}{2}$. The edge set of $T_k$ has the following elements: $v$ is adjacent to every vertex in $B$ and for every subset $\{a_i,a_j\}$ of size $2$ in $A$, there is exactly one vertex $b_{ij} \in B$ whose neighborhood in $A$ is precisely $\{a_i,a_j\}$. See Figure \ref{fig:negative3Deg} for an illustration of $T_2$.
 \begin{figure}[!ht]
\centering
\resizebox{0.3\textwidth}{!}{%
\begin{circuitikz}
\tikzstyle{every node}=[font=\LARGE]
\draw [dashed, line width=0.8pt ] (3.75,7.25) ellipse (1cm and 2.75cm);
\draw [dashed, line width=0.8pt ] (7.25,7.25) ellipse (1cm and 3.75cm);
\draw [ fill={rgb,255:red,0; green,0; blue,0} , line width=0.2pt ] (3.75,8.75) circle (0.15cm);
\draw [ fill={rgb,255:red,0; green,0; blue,0} , line width=0.2pt ] (3.75,7.75) circle (0.15cm);
\node [font=\LARGE] at (4.25,7.25) {};
\node [font=\LARGE] at (4.25,7.25) {};
\draw [ fill={rgb,255:red,0; green,0; blue,0} , line width=0.2pt ] (3.75,6.75) circle (0.15cm);
\node [font=\LARGE] at (4.25,6.25) {};
\node [font=\LARGE] at (4.25,6.25) {};
\node [font=\LARGE] at (4.25,6.25) {};
\draw [ fill={rgb,255:red,0; green,0; blue,0} , line width=0.2pt ] (3.75,5.75) circle (0.15cm);
\draw [ fill={rgb,255:red,0; green,0; blue,0} , line width=0.2pt ] (7.25,9.75) circle (0.15cm);
\draw [ fill={rgb,255:red,0; green,0; blue,0} , line width=0.2pt ] (7.25,8.75) circle (0.15cm);
\draw [ fill={rgb,255:red,0; green,0; blue,0} , line width=0.2pt ] (7.25,7.75) circle (0.15cm);
\draw [ fill={rgb,255:red,0; green,0; blue,0} , line width=0.2pt ] (7.25,6.75) circle (0.15cm);
\draw [ fill={rgb,255:red,0; green,0; blue,0} , line width=0.2pt ] (7.25,5.75) circle (0.15cm);
\draw [ fill={rgb,255:red,0; green,0; blue,0} , line width=0.2pt ] (7.25,4.75) circle (0.15cm);
\draw [ fill={rgb,255:red,0; green,0; blue,0} , line width=0.2pt ] (10.5,7.25) circle (0.15cm);
\draw [line width=0.8pt, short] (7.25,9.75) -- (10.5,7.25);
\draw [line width=0.8pt, short] (7.25,8.75) -- (10.5,7.25);
\draw [line width=0.8pt, short] (7.25,7.75) -- (10.5,7.25);
\draw [line width=0.8pt, short] (7.25,6.75) -- (10.5,7.25);
\draw [line width=0.8pt, short] (7.25,5.75) -- (10.5,7.25);
\draw [line width=0.8pt, short] (7.25,4.75) -- (10.5,7.25);
\draw [line width=0.8pt, short] (3.75,8.75) -- (7.25,9.75);
\draw [line width=0.8pt, short] (3.75,7.75) -- (7.25,9.75);
\draw [line width=0.8pt, short] (3.75,8.75) -- (7.25,8.75);
\draw [line width=0.8pt, short] (3.75,6.75) -- (7.25,8.75);
\draw [line width=0.8pt, short] (3.75,8.75) -- (7.25,7.75);
\draw [line width=0.8pt, short] (3.75,5.75) -- (7.25,7.75);
\draw [line width=0.8pt, short] (3.75,7.75) -- (7.25,6.75);
\draw [line width=0.8pt, short] (3.75,6.75) -- (7.25,6.75);
\draw [line width=0.8pt, short] (3.75,7.75) -- (7.25,5.75);
\draw [line width=0.8pt, short] (3.75,5.75) -- (7.25,5.75);
\draw [line width=0.8pt, short] (3.75,6.75) -- (7.25,4.75);
\draw [line width=0.8pt, short] (3.75,5.75) -- (7.25,4.75);
\node [font=\LARGE] at (3.5,11.75) {};
\node [font=\LARGE] at (3.75,10.75) {$A$};
\node [font=\LARGE] at (7.25,11.5) {$B$};
\node [font=\LARGE] at (10.5,8) {$v$};
\end{circuitikz}
}%
\caption{The graph $T_2$.\label{fig:negative3Deg}}
\end{figure}
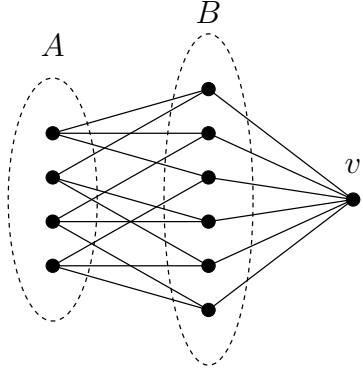

    Any vertex in this graph dominates at most $2$ vertices from $A$, which gives $\gamma(T_k) \geq k$.
    Considering the packing number, observe that any two vertices in $A$ have a (unique) common neighbor in $B$, so we can have at most one vertex from $A$ in the packing. Similarly, any  two vertices in $B$ are adjacent to $v$, so we can have at most one vertex from $B \cup \{v\}$ in the packing. This shows that $\rho(T_k)\le 2$. 
\end{proof}

\subsection{2-degenerate graphs}
For $2$-degenerate graphs, we cannot require to have a constant ratio $(X,Y)$-dominating-packing pair for all $X,Y \subseteq V(G)$. To see this, take the counterexample $T_k$ described for $3$-degenerate graphs (in the proof of Theorem~\ref{th:negative3deg}) and consider $T'_k=T_k\setminus \{v\}$, which is $2$-degenerate. It is easy to verify that with the choice of $Y = B$ and $X= \emptyset$, we get $\rho_{X,Y}(T'_k) = 1$ and  $\gamma_{X,Y}(T'_k) \geq  k = k \cdot \rho_{X,Y}(T'_k)$.

Therefore, we forbid the use of $Y$. 
This leads to a more technical induction hypothesis, including the following sets: for any set $X \subseteq V(G)$, we define $X^{2+}=\{ x \in X ~:~ \deg_G(x) \geq 2\}$ and $X^1=\{ x \in X ~:~ \deg_G(x) =1\}$.  We will prove the following statement.

\begin{theorem}\label{th:2deg}
    For every $2$-degenerate graph $G$ and for every set $X\subseteq V(G)$, we have 
    \begin{equation}\label{eq:2-deg-hypothesis}
        \gamma_{X,\emptyset}(G) \leq 7 \rho_{X,\emptyset}(G) + 2\cdot |X^{2+}| + |X^1|.
    \end{equation}
     In particular, for $X = \emptyset$, we obtain $\gamma(G) \leq 7 \rho(G)$.
\end{theorem}

\begin{proof}

We proceed by contradiction, letting $G=(V,E)$ be a subgraph-minimal graph for which 
 there exists a set $X \subseteq V$ with $\gamma_{X,\emptyset}(G) > 7 \rho_{X,\emptyset}(G) + 2\cdot |X^{2+}| + |X^1|$. Note that by the minimality of $G$, $G$ must be connected. We continue by deducing some properties of $(G,X)$.
\begin{claim}\label{2deglemma} By the minimality of $G$,
\begin{enumerate}
    \item[i)] any $u \in V \setminus X$ has at least two neighbors in $V \setminus X$;
    \item[ii)] any $u \in N(X)$ has at least two neighbors in $V \setminus N[X]$;
    \item[iii)] any $u \in V \setminus X$ has at most one neighbor in $X$;
    \item[iv)] $X$ is a stable set;
    \item[v)] any $u \in V\setminus X$ has degree at least $3$.
\end{enumerate}
\end{claim}
\begin{proof} The proof of each part follows a similar argument. We assume that the property does not hold and then show the existence of $G'\subsetneq G$ and $X' \subseteq V(G')$ such that any dominating set $D'$ and packing $P'$ for $(G',X')$  satisfying 
\begin{equation}\label{eq:2-deg-Gprime}
     |D'| \leq 7|P'| + 2|(X')^{2+}| + |(X')^{1}|
\end{equation}
can be used to construct a dominating set $D$ and a packing $P$ for $(G,X)$  which satisfies
\begin{equation}\label{eq:2-deg-G}
     |D| \leq 7|P| + 2|(X)^{2+}| + |(X)^{1}|\, ,
\end{equation}
reaching contradiction. Note that the pair $(D',P')$ satisfying Equation~\eqref{eq:2-deg-Gprime}  always exists because of the minimality of $G$. 
    \begin{enumerate}
        \item[i)] Assume that there exists a vertex $u \in V \setminus X$ with at most $1$ neighbor in $V\setminus X$. We distinguish two cases.
            \begin{enumerate}
                \item[a)] \textit{If $u$ has a neighbor in $X$}, consider the pair $(G', X') = (G \setminus \{u\}, X)$ and observe that any dominating-packing pair $(D',P')$ for $(G',X')$ is also a valid pair for $(G,X)$, because  $u$ is in $N(X)$ and it cannot be a common neighbor of two vertices in $P' \subseteq V \setminus X$. It also satisfies Equation~\eqref{eq:2-deg-G} as $X = X'$.
                \item[b)] \textit{If $u$ does not have a neighbor in $X$}, then let $v \in V\setminus X$ be the neighbor of $u$ and define $(G',X') =(G\setminus\{u\}, X \cup \{v\})$. Let ($D',P')$ be an $(X',\emptyset)$-dominating-packing pair that satisfies Equation~\eqref{eq:2-deg-Gprime}. Observe that $D= D'\cup \{v\}$ and $P= P' \cup \{u\}$ are valid dominating and packing sets for $(G,X)$, 
                with
                $$|D| 
                = 
                |D'|+1 
                \leq 
                7|P'| + 2|(X ')^{2+}| + |(X ')^{1}| +1 
                \leq
                 7|P| + 2|X^{2+}| + |X^{1}|.$$
                 
            \end{enumerate}
        \item[ii)]  Assume that there exists a vertex $u \in N(X)$ with at most $1$ neighbor in $V\setminus N[X]$.\\ In this case, we consider $(G',X')= (G \setminus \{u\}, X)$. It is easy to see that any dominating-packing pair $(D',P')$ for $(G',X')$  is  a valid pair for $(G,X)$ as well, because $u \in N(X)$ and it cannot have more than $1$ neighbor belonging to $P'$. Furthermore, $(D,P) = (D',P')$  satisfies Equation~\eqref{eq:2-deg-G} as $X = X'$.
        \item[iii)]  Assume that there exists a vertex $u \in V\setminus X$ with at least $2$ neighbors in $X$.\\
        Let $x_1,x_2 \in X$ be two distinct neighbors of $u$ and consider $G' = (V,E\setminus \{ux_1\})$ and $X' = X$. Let $(D',P')$ be a dominating-packing pair for $(G',X')$ which satisfies \Cref{eq:2-deg-Gprime}. It is easy to see that $(D,P) = (D',P')$ is a valid pair for $(G,X)$ as well that satisfies Equation~\eqref{eq:2-deg-G} as $X = X'$.
        \item[iv)]  Assume that there are vertices $x_1,x_2 \in X$  such that $x_1x_2 \in E$.\\
        Let  $G' = (V,E\setminus \{x_1x_2\})$ and $X' = X$. Once again, any dominating-packing pair for $(G',X')$ is also a valid pair for $(G,X)$, which satisfies Equation~\eqref{eq:2-deg-G} as $X = X'$.
       \item[v)]  Assume that there exists a vertex $u\in V\setminus X$ with $\deg(u) \leq 2$.\\
       Part i) implies that $\deg(u) = 2$ and $N(u) \cap X = \emptyset$. Let $(G',X') = (G \setminus \{u\}, X\cup N(u))$ and let $(D',P')$ be a dominating-packing pair for $(G',X')$. \\
       We define $(D,P) = (D' \cup N(u), P'\cup\{u\})$, which is clearly a dominating-packing pair for $(G,X)$ with $|P|=|P'|+1$ and $|D|=|D'|+2$. We also have $|X'|=|X|+2$  which implies $|(X')^1| + 2|(X')^{2+}| \leq |X^1| + 2|X^{2+}| +4$. Putting everything together, we obtain
        \[
            |D|
            =
            |D'|+2
            \leq 
            7|P'| + 2|(X')^{2+}| + |(X')^{1}|+2
            \leq 
             7|P| + 2|X^{2+}| + |X^{1}|.
       \]       
    \end{enumerate}
\end{proof}

\noindent
We proceed by considering the following three sets.
\begin{align*}
    A_1 &= \{ u \in V ~:~ \deg_G(u) \leq 2\};\\
    A_2 &= \{ u \in V\setminus A_1 ~:~ \deg_{G\setminus A_1}(u) \leq 2\};\\
    A_3 &= \{ u \in V\setminus (A_1\cup A_2) ~:~ \deg_{G\setminus (A_1\cup A_2)}(u) \leq 2\}.
\end{align*}
Now we put \Cref{2deglemma} into action to show some properties of these sets. 
By \Cref{2deglemma}~v), we have that $A_1 \subseteq X$. Since each vertex in $A_2$ has at least one neighbor in $A_1 \subseteq X$, we have
 $A_1 \cup A_2 \subseteq N[X]$.
As $A_2 \subseteq N(X)$,  \Cref{2deglemma} iv) implies that $A_2\subseteq V \setminus X$. From the definition of $A_2$ and \Cref{2deglemma}~iii), we can deduce that $\deg_G(u) = 3$ for all $u \in A_2$. As $A_2 \subseteq N(X)$, \Cref{2deglemma}~ii) gives that any vertex of $A_2$ has exactly two neighbors in $V \setminus N[X] $ and one neighbor in $A_1 \subseteq X$.

 If  $A_3 = \emptyset$, then  by the $2$-degeneracy of $G$, we have  $G\setminus (A_1 \cup A_2) = \emptyset$. In this case, since $A_1 \cup A_2 \subseteq N[X]$, we get that $\gamma_{X,\emptyset}(G) = 0$ and so Equation~\eqref{eq:2-deg-hypothesis} holds for $(G,X)$, a contradiction. Therefore we can assume that there exists a vertex $u \in A_3$. Note that $u$ must have a neighbor in $A_2$ and as each vertex in $A_2$ has exactly two neighbors in $V \setminus N[X] $ and one neighbor in $A_1$, we have $u \in V \setminus N[X]$. 
Let $v_1$ be a neighbor of $u$ in $A_2$ and $x_1$ be the neighbor of $v_1$ in $A_1$.

\begin{claim}
    The vertex $u$ has exactly two neighbors in $ V\setminus N[X]$.
\end{claim}
\begin{proof}
As $A_1 \cup A_2 \subseteq N[X]$, we have $\deg_{G\setminus N[X]}(u) \leq \deg_{G\setminus (A_1\cup A_2)}(u) \leq 2$. By contradiction, if  $u$ has at most $1$ neighbor in $V\setminus  N[X]$, then consider a dominating set $D'$ and a packing  $P'$ for $(G\setminus \{x_1,v_1,u\}, X \setminus \{x_1\})$ satisfying
$
     |D'| \leq 7|P'| + 2|X^{2+}| + |X^{1}|-1.
$
Observe that $D = D' \cup \{u\}$ and $P=P'$ is a valid dominating-packing pair for $(G,X)$ satisfying \Cref{eq:2-deg-G}, a contradiction.
\end{proof}
\begin{figure*}[h]\centering
    \includegraphics[width=7cm]{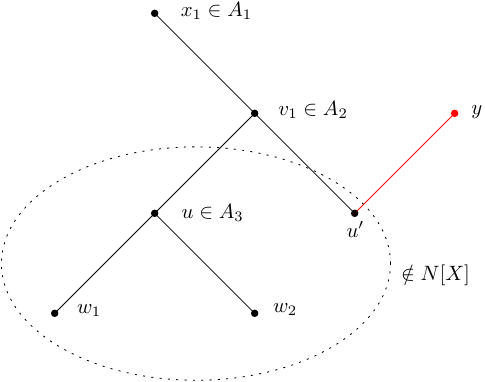}
    \end{figure*}
Let $w_1,w_2 \in V\setminus N[X]$ be two distinct neighbors of $u$. As $v_1 \in A_2$, we know that $v_1$ has exactly two neighbors in $V\setminus N[X]$, one of these is $u$ and we denote the other one by $u'$. 
Consider now the following pair:
\[
G' = \big(V \setminus \{u,v_1,x_1\}\cup \{y\}, E|_{V \setminus \{u,v_1,x_1\}} \cup \{u'y\}\big) \quad \text{ and } \quad X' = X \setminus \{x_1\}\cup \{w_1,w_2,y\}.
\]
Let $D'$ and $P'$ be a dominating set and a packing for $(G',X')$  satisfying 
\begin{equation*}
     |D'| \leq 7|P'| + 2|(X')^{2+}| + |(X')^{1}|.
\end{equation*}
Let $N_G^2(u)$ denote the second neighborhood of $u$, that is, the set of vertices that are at distance at most $2$ from $u$. Since $N_G^2(u)\setminus \{x_1,v_1\} \subseteq N_{G'}[X']$, $P'$ cannot contain any vertex in $N_G^2(u)$ and so $P=P' \cup \{u\}$ is a packing in $G$. In order to extend the dominating set, we define $D = D' \cup \{w_1,w_2,v_1\}$. Since $y$ has degree $1$ (and $x_1,w_1,w_2$ are of arbitrary degree), we have $2|(X')^{2+}| + |(X')^{1}| \leq 2\left(|X^{2+}|+2\right) + \left(|X^{1}|+1\right) -1 = 2 |X^{2+}| +|X^{1}|+4 $, which gives
\begin{align*}
|D|
&= |D'|+3
\leq 7|P'| + 2|(X')^{2+}| + |(X')^{1}| +3
\leq 7|P'| +  2|X^{2+}| + |X^{1}|+4+3\\
&= 7(|P'|+1) +  2|X^{2+}| + |X^{1}|
= 7|P| +  2|X^{2+}| + |X^{1}|,
\end{align*}
that is, $(D,P)$ is a dominating-packing pair for $(G,X)$ that satisfies \Cref{eq:2-deg-G}, a contradiction. This concludes the proof of \Cref{th:2deg}.
\end{proof}

\section{Additional families of graphs}\label{sec:additional}

\subsection{Asteroidal triple-free graphs}

An {\it asteroidal triple} in a graph $G$ is a set $T$ of $3$ vertices such that $G[T]$ is independent and every pair of vertices in $T$ can be joined by a path that avoids the neighborhood of the third vertex in $T$. A graph $G$ is {\it asteroidal triple-free (AT-free)} if $G$ does not contain an asteroidal triple. Various classes of graphs are known to be AT-free, for example, interval graphs, co-comporability graphs and permutation graphs (see e.g. \cite{COS97}). We show that all these classes have a bounded ratio by proving the following theorem.

\begin{restatable}{theorem}{ATfree}
    \label{th:ATfree}
    For every asteroidal triple-free graph $G$, we have $\gamma(G) \leq 3 \cdot \rho(G)$.
\end{restatable}

\begin{proof}
Our proof is based on  the following property.

\begin{theorem}\cite{COS97}\label{ATfreeDom}
    Every connected asteroidal triple-free graph contains a pair of vertices such that any path between this pair of vertices is dominating. 
\end{theorem}
    Let $G$ be an AT-free graph. By Theorem \ref{ATfreeDom} there is a pair of vertices $u,v\in V(G)$ such that any path between $u$ and $v$ is dominating. Let $\Pi=\pi_1,\pi_2,\ldots, \pi_p$ be the shortest path between $u$ and $v$. Thus, we have $\gamma(G)\le p$. On the other hand as $\Pi$ is a shortest path in $G$, by taking the set of vertices $\{\pi_{3i+1}\mid 0\le i \le \lfloor p/3 \rfloor-1 \}$ we get a packing in $G$. Therefore, $\rho(G)\ge \lfloor p/3 \rfloor \geq \lfloor \gamma(G)/3 \rfloor$. 
\end{proof}

\subsection{Convex graphs}

A bipartite graph $G = (X \cup Y, E)$ is \emph{convex} if and only if one of the color classes, without loss of generality $X$, has an ordering such that for all $ y \in Y$ the vertices adjacent to $y$ are consecutive in the ordering. We refer to this ordering as {\em convex ordering}. 

\begin{theorem}
    For every convex graph $G$, we have $\gamma (G) \leq 3 \cdot \rho (G)$.
\end{theorem}
\begin{proof}
    We can assume that $G$ does not have an isolated vertex.
    We will encode convex graphs using points and intervals in $\mathbb R$. In particular, if the convex ordering of $X$ is $x_1, \dots, x_n$, we use the embedding $\pi : X \to \mathbb R$ defined as $\pi (x_i) = i$, and for each $y \in Y$, we associate the interval $\iota(y) = \{ \pi(x) : x \in N_G(y)\} $, see \Cref{fig:convex} for an illustration. For subsets $X' \subseteq X$ and $Y' \subseteq Y$, we use the notation
    $
    \pi(X') =\big\{ \pi(x)  ~:~ x \in X'\big\}$ and $  \iota(Y') = \big\{ \iota (y) ~:~ y \in Y'\big\}\,.
    $
    In this encoding, the image of a set $S \subseteq V(G)$ is a union of points and intervals, corresponding to the vertices in $S \cap X$ and $S \cap Y$, respectively. 
    \begin{figure}[h]
    \includegraphics[width=\textwidth]{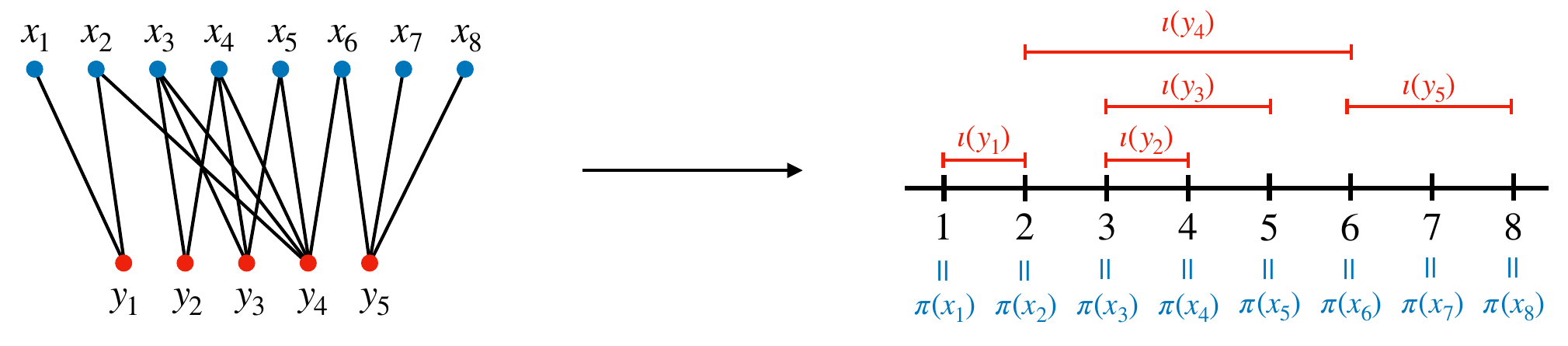}
    \caption{The embeddings $\pi$ and $\iota$.}\label{fig:convex}
    \end{figure}
    \newpage
    \noindent
    Observe that a set $D \subseteq V(G)$ is a dominating set if and only if
    \begin{enumerate}
        \item[(d1)] $\iota(D \cap Y)$ covers $\pi(X \setminus D)$, and
        \item[(d2)] $\pi( D \cap X)$ hits all the intervals in $\iota(Y\setminus D)$.
    \end{enumerate}
    Moreover, a set $P \subseteq V(G)$ is a packing if and only if the intervals in $\iota(P\cap Y)$ are disjoint,
         any interval in $\iota (Y)$ contains at most one point from $\pi(P \cap X)$,
        and no point from $\pi(P \cap X)$ is contained in any interval in $\iota(P\cap Y)$.
    Let $P$ be a maximal packing that minimizes the sum of the length of intervals in $\iota(P \cap Y)$, that is, the sum of degrees of vertices in $P\cap Y$. We construct a dominating set $D = P \cup D_1 \cup D_2$ with
    \begin{itemize}
        \item $D_1 = \big\{ x_y^1, x_y^2 ~:~ y \in P \cap Y \big\}\subseteq X$, where $x_y^1, x_y^2$ are the vertices corresponding to the end-points of the interval $\iota(y)$; 
        \item $D_2 = \big\{ y_x^1,y_x^2 ~:~ x \in P \cap X\big\}\subseteq Y$, where $\iota(y_x^1)$ is the interval that contains $\pi(x)$ and starts the earliest and $\iota(y_x^2)$ is the interval that contains $\pi(x)$ and ends the latest.
    \end{itemize}
    Note that $D$ has size at most $|P| + 2|P \cap X| + 2 |P \cap Y| = 3|P|$.
    First, we show that $D$ satisfies Property (d1). As $P \subseteq D$, it is sufficient to argue that $\iota(D \cap Y)$ covers $\pi(X\setminus P)$. By the maximality of $P$,  any point $p \in \pi(X \setminus P)$ is either contained in an interval of $\iota(P \cap Y)$, or there is an interval in $\iota(Y)$ that contains both $p$ and a point $\pi(x) \in \pi(P \cap X)$. As $\iota(P \cap Y) \subseteq \iota(D \cap Y)$, we only need to discuss the second case. Observe that by the definition of $y_x^1,y_x^2 \in D_2$, the union of the intervals $\iota(y_x^1)$ and $\iota(y_x^2)$ covers every point that is contained in any interval that contains $\pi(x)$, therefore $p \in \iota(D_2\cap Y) \subseteq \iota(D\cap Y)$.

    We follow a similar argument to show that $D$ has Property (d2). Again, as $P \subseteq D$, it is sufficient to show that $\pi(D \cap X)$ hits all the intervals in $\iota(Y \setminus P)$.
    By the maximality of $P$, for any $I \in \iota( Y \setminus P)$, either there is a point $p \in \pi(P \cap X)$ such that $p \in I$, in which case we are done, or there is an interval $\iota(y) \in \iota(P\cap Y)$ such that $I \cap \iota(y) \neq \emptyset$. In the second case, observe that by the minimality of the sum of the length of the intervals in $P$, we have $I \nsubseteq \iota(y)$ and so $I$ must contain at least one of the points $\pi(x_y^1), \pi(x_y^2) \in \pi(D_1) \subseteq \pi(D \cap X)$.    
\end{proof}

\subsection{Geometric graphs}

One of the most general families of geometric graphs is the family of string graphs. 
A {\it string graph} is the intersection graph of a finite collection of curves\footnote{ A {\it curve} (“string”) is a subset of the plane which is homeomorphic to the interval $[0,1]$.} in the plane. However, this class is too broad to have bounded ratio. To see this, we can use the facts that any chordal graph can be realized as a string graph and that split graphs are chordal graphs. Thus, Theorem \ref{th:negativesplit} implies the following corollary.

\begin{corollary}
    For every $k \in \NN$, there is a string graph $S_k$ such that $\rho(S_k) = 1$ and $\gamma(S_k) = k$.
\end{corollary}

On the other hand, we can obtain a bounded ratio for more constrained geometric intersection graphs. For instance, for {\it unit disk graphs}, that are intersection graphs of disks of radii $1$ in the plane, we obtain the following result. 

\begin{theorem}
    For any unit disk graph $G$, we have $\gamma(G)\le 32 \cdot \rho(G)$. 
\end{theorem}

\begin{proof}
We will use the following claim based on elementary geometry.
\begin{claim}\label{disjointDisks}
    For any fixed unit disk $d\in V(G)$, there is a set of at most $32$ points in the plane that intersect each of the disks in $N[N[d]]$. 
\end{claim}
\begin{proof}
    The disks in $N[N[d]]$ are contained in a disk $D$ of radius $5$ centered in the center of $d$. By \cite{V49} this disk can be covered by a set $S$ of at most $32$ unit disks. The collection of the centers of these unit disks intersect each of the disks contained in $D$, as otherwise $S$ is not a covering of $D$. 
\end{proof}    

\noindent
Let $P\subseteq V(G)$ be a maximum packing in $G$. Observe that the maximality of $P$ implies that $\cup_{d\in P}N[N[d]] = V(G)$.
To find a suitable dominating set, for each $d \in P$ we apply \Cref{disjointDisks} and for each of the $32$ points, we pick a disk in $V(G)$ that contains it (if such exists). In this way, we obtain a set $S(d)\subseteq V(G)$ of at most $32$ disks that dominates $N[N[d]]$. Taking the union of $S(d)$ over all $d \in P$, we obtain a dominating set in $G$ of size at most $32\cdot |P|$, as required.
\end{proof}

\section{Negative results}\label{sec:negative}

\subsection{A counterexample for Conjecture~\ref{conj:deg3}}
Recall the statement of the conjecture.

\maxdegconj*

\noindent
In this section, we show that the second part of Conjecture~\ref{conj:deg3} is false by describing an infinite family of graphs of maximum degree at most $3$ and $\gamma(G)=2\rho(G)+1$.

\negativeConj*

\begin{proof}
    First, we describe the basic building blocks of the graph class and an operation called \textit{chaining} that is used to generate the graphs in $\mathcal G$.
    
    A {\it block} is a graph on six vertices consisting of a path $P_6=v_1,v_2,\ldots,v_6$ together with two additional edges $\{v_1,v_5\}$ and $\{v_2,v_6\}$, see Figure \ref{fig:block} for an illustration. We refer to the vertices $\{v_3,v_4\}$ as {\it level one} vertices, the vertices $\{v_2,v_5\}$ as {\it level two} vertices and to $\{v_1,v_6\}$ as {\it level three} vertices of the block. The operation of {\it chaining} two blocks $A,B$ on vertex sets $\{a_1,a_2,\ldots,a_6\}$ and $\{b_1,b_2,\ldots,b_6\}$ returns the graph equal to the disjoint union of $A$ and $B$ with the additional edges $\{a_1,b_3\}$ and $\{a_6,b_4\}$. 
    \begin{figure}[h]\centering
    \includegraphics[width=5cm]{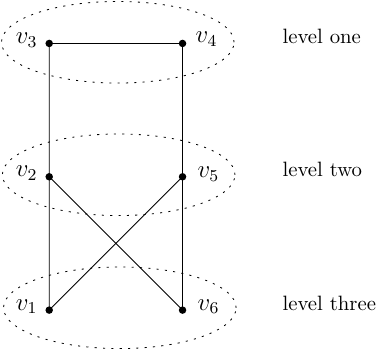}
    \caption{A block with its partition into levels. \label{fig:block}}
    \end{figure}
    
    Now we are ready to describe the graph class $\mathcal G = \{G_i ~|~ i \in \NN\}$. For any $i \in \NN$, the graph $G_i$ is obtained by taking $i$ blocks $A_1, A_2, \dots, A_i$, performing the chaining operation on the pairs $A_j,A_{j+1}$ for $j = 1, \dots, i-1$, then adding an edge $e=\{u_1,u_2\}$ and `chaining' the pairs $A_i,e$ and $e,A_1$. Formally, in the last step, we add the edges $\{u_1,a^1_3\},\{u_2,a^1_4\}$ and $\{u_1,a^i_1\},\{u_2,a^i_6\}$ where the vertices of $A_j$ are labeled as $ \{a^j_1,a^j_2,\ldots,a^j_6\}$ for each $j \in [i]$.  See Figure \ref{fig:blockcycle} for an illustration of $G_3$.
     \begin{figure}[h]\centering
    \includegraphics[width=9cm]{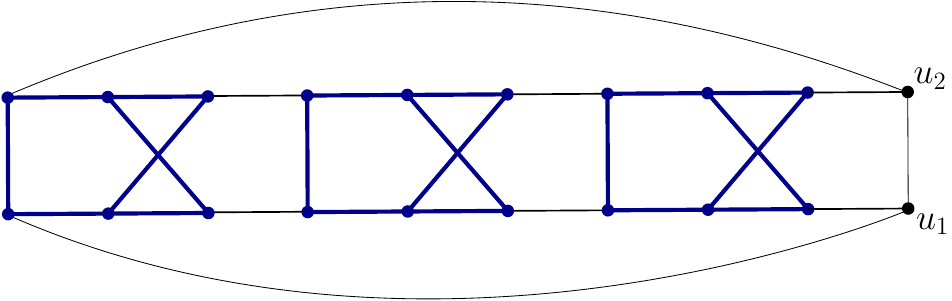}
    \caption{The graph $G_3$ obtained by chaining three blocks and an edge. The blocks are highlighted by thick, blue edges. \label{fig:blockcycle}}
    \end{figure}

    \begin{claim}
  For each $i \in \NN$, $\rho(G_i)=i$.
    \end{claim}
   \begin{proof} We can construct a packing of size $i$ in $G_i$ by fixing a level $\ell \in \{1,2,3\}$ and picking a level $\ell$ vertex from each block of $G_i$. 
   
   Assume to the contrary that there is a packing $P$ of size at least $i+1$. Observe that $P$ cannot contain two vertices from the same block, therefore $P$ must contain exactly one vertex from each block and one of $\{u_1,u_2\}$. 
   This implies that $P$ must contain a level $3$ vertex in its first block, forcing $P$ to have a level $3$ vertex in its second block, and so on until the $i^{\mathrm{th}}$ block, where we reach an impossibility of $P$ containing one of $\{u_1,u_2\}$  as well as a level $3$ vertex of the last block.
\end{proof}

 \begin{claim}
  For each $i \in \NN$, $\gamma(G_i)=2i+1$.
    \end{claim}
   \begin{proof}
    We can construct a dominating set of size $2i+1$ in $G_i$ by fixing both level $\ell \in \{1,2,3\}$ and picking the level $\ell$ vertices from each block plus one of $\{u_1,u_2\}$.
   
     Assume to the contrary that there is a dominating set $D$ of $G_i$ of size at most $2i$. Let $A_j$  be a block of $G_i$ with $V(A_j) =  \{v_1,v_2,v_3,v_4,v_5,v_6\}$. First observe that in order to dominate the level $2$ vertices of $A_j$, we must have $|D\cap V(A_j)| \geq 1$. 
     
     If $A_j$ contains exactly $1$ vertex from $D$ then it must be one of the vertices from level three, otherwise either $v_2$ or $v_5$ could not be dominated. This also implies that $|V(A_{j-1})\cap D| \geq 2$ and if the block $A_{j-1}$ contains exactly $2$ vertices from $D$, then it must be the ones on level three, otherwise either $v_3$ or $v_4$ could not be dominated. Similarly, in this case $|V(A_{j-2})\cap D| \geq 2$ and if the block  $A_{j-2}$ has exactly $2$ vertices from $D$, both of them are from level three. Hence we can deduce the following observation. 

     \begin{observation}\label{obs:blocksD}
         Let $D$ be a dominating set in $G_i$, then there is no sub-chain of blocks such that all the internal blocks in the sub-chain contain exactly $2$ vertices from $D$ and the extremal blocks (one of which can be the edge $\{u_1,u_2\}$) contain at most $1$ vertex from $D$. 
     \end{observation}

    \noindent
    We will reach contradiction with $\gamma(G_i) \leq 2i$ via a case distinction on $|D\cap \{u_1,u_2\}|$.
    
     Assume that $|D\cap \{u_1,u_2\}|=2$. Then there are at least two blocks with exactly one vertex from $D$ such that all the blocks between them in the chaining have at most $2$ vertices from $D$, but this is a contradiction to \Cref{obs:blocksD}.
     
     Assume that $|D\cap \{u_1,u_2\}|=1$. Then there is at least one block with exactly one vertex from $D$. If there is exactly one such a block then the rest of the blocks contain exactly $2$ vertices from $D$ and we get a contradiction to \Cref{obs:blocksD}. If there are more than two such blocks then we again reach a contradiction to \Cref{obs:blocksD}.

     Finally, assume that $|D\cap \{u_1,u_2\}|=0$. If there is a block with exactly one vertex from $D$ then we also get a contradiction to \Cref{obs:blocksD}.
     Hence each block has exactly two vertices in $D$. 
    We consider three cases: the vertices $\{u_1,u_2\}$ are dominated by a pair of vertices in level one of the first block; the vertices $\{u_1,u_2\}$ are dominated by a pair of vertices in level three of the last block; the vertices $\{u_1,u_2\}$ are dominated by a vertex in level one and a vertex in level three of the first and last blocks respectively. In the first case, in each of the blocks we must choose the pair of vertices in level one to $D$ and so the level three vertices in the last block are not dominated, a contradiction. In the second case, similarly, $D$ must contain the pair of vertices in level three in each of the blocks, and so the level one vertices in the first block are not dominated, a contradiction. Finally, in the third case, $D$ must contain at least one vertex from the third level in each block, but this leaves at least one vertex in the first block not dominated, a contradiction. Hence we conclude the proof of the claim.
    \end{proof}
This completes the proof of \Cref{thm:negativeConj}.    
\end{proof}

\subsection{Split graphs}

A graph is called a \textit{split graph} if its vertex set can be partitioned into two sets such that one of them induces a clique and the other one is an independent set.

\negativeSplit*

\begin{proof}
For a fixed $k \in \NN$, we define the graph $S_k$ as follows 
\begin{itemize}
    \item $V(S_k) = C ~\dot\cup~ I$ with $|C|=2k-1$ and $|I|=\binom{2k-1}{k}$;
    \item $C$ induces a clique, $I$ is an independent set and to every subset of size $k$ in $C$, we associate exactly one vertex in $I$ which is adjacent to precisely the vertices of this subset, see Figure \ref{fig:negativeSplit} for an illustration of $S_2$.
\end{itemize}

     \begin{figure}[h]\centering
    \includegraphics[width=5cm]{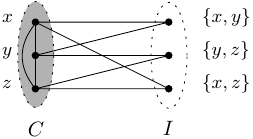}
    \caption{The graph $S_2$.\label{fig:negativeSplit}}
    \end{figure}

    First, we show that $\gamma(S_k) = k$. 

    To see that $\gamma(S_k) \leq k$, we will argue that any set $D$ of $k$ vertices from $C$ is a dominating set. Indeed, the vertices in $C$ are dominated, since $C$ is a clique. Any vertex $v \in I$ is also dominated, since it has $k$ neighbors in $C$ and $|C\setminus D|=k-1$, therefore the neighborhood of $v$ cannot be disjoint from $D$.
    To show that $\gamma(S_k) \geq k$, let $D'$ be any set of size $k-1$ and let $i \in [0,k-1]$ denote $|D' \cap I|$. Using this notation, we have $|D' \cap C| = k-1-i$ and $|C\setminus D'| = k+i$, which implies that $\binom{k+i}{k} > i$ vertices in $I$ are not dominated by $D' \cap C$. This set of at least $i+1$ vertices in $I$ should be dominated by the $i$ vertices in $D' \cap I$, which is impossible since $I$ is stable.

    It is easy to see that any pair of vertices in $S_k$ are either adjacent or have a common neighbor, which implies that $\rho(S_k) \leq 1$ and $\rho(S_k) = 1$ is trivially attainable.
\end{proof}

\section{Conclusion}\label{sec:ccl}
In this work, we have made progress in resolving the following major open problem.

\begin{question}
    Characterize those graph classes $\mathcal G$ for which there is a constant  $c_{\mathcal G}$ such that \[\frac{\gamma(G)}{\rho(G)}\le c_{\mathcal{G}} \quad \text{ for each } \quad G\in \mathcal{G} ~~ .\] 
\end{question}

We extended the lists of both positive and negative examples of graph classes. Moreover, we note that due to Theorem \ref{th:2deg} on $2$-degenerate graphs we cannot hope for a characterization of the above graph classes as sparse graph classes. 
We also improved the constants $c_{\mathcal G}$ for several important graph classes and sharpened Conjecture \ref{conj:deg3}.
We finish our paper by stating some open questions.

From previous works, the complete resolution of Conjecture \ref{conj:deg3}~\cite{HLR11} and bounding the ratio $\gamma/\rho$ for homogeneously orderable and chordal bipartite graphs~\cite{GG24} remain open.

The question of bounding the ratio $\gamma/\rho$ for graphs in many geometric graph classes remains open as well. For example, intersection graphs of axis-parallel rectangles in the plane or intersection graphs of axis-parallel boxes in three dimensions. Note that, using the result in \cite{AS14}, a bound of the form $\gamma(G)\le (\Delta(G)-1)\rho(G)+1$ in the latter class of graphs would imply Conjecture \ref{conj:deg3}.

Another direction for future work is improving the known bounds, for example, we believe that our bounds on planar, unit-disk, and $2$-degenerate graphs could be improved. 

\section{Acknowledgements}
The second and fourth authors learned about the question studied in this paper at the 11th Annual Workshop on Geometry and Graphs held at Bellairs Research Institute in March 2024. They are grateful to the organizers for the unique opportunity to attend this workshop. 
Yelena Yuditsky was supported by FNRS as a Postdoctoral Researcher.

\bibliographystyle{alpha}
\bibliography{coolnew}

\end{document}